\documentclass[final]{siamart0216}
\usepackage{amsfonts,amsmath,amssymb}
\usepackage{mathtools}
\usepackage{enumerate}
\usepackage{xspace} 
\usepackage{hyperref}
\usepackage{tikz}
\everymath{\displaystyle}

\DeclareMathOperator{\dist}{dist}
\newcommand{\R}{\mathbb{R}}
\newcommand{\bu}{\mathbf{u}}
\newcommand{\bv}{\mathbf{v}}
\newcommand{\GRAD}{\nabla}
\DeclareMathOperator{\DIV}{div}
\newcommand{\diff}{\, \mbox{\rm d}}

\newcommand{\bX}{{\mathbf{X}}}
\newcommand{\bef}{{\mathbf{f}}}
\newcommand{\bg}{{\mathbf{g}}}
\newcommand{\bH}{{\mathbf{H}}}
\newcommand{\bL}{{\mathbf{L}}}
\newcommand{\bn}{{\mathbf{n}}}
\newcommand{\bx}{{\mathbf{x}}}
\newcommand{\bU}{{\mathbf{U}}}
\newcommand{\calO}{{\mathcal{O}}}
\newcommand{\calA}{{\mathcal{A}}}

\newsiamremark{remark}{Remark}
\newcommand{\ermk}{\hfill\ensuremath{\blacksquare}}
\newcommand{\TheTitle}{The Darcy problem with porosity depending exponentially on the pressure}
\newcommand{\ShortTitle}{Exponential porosity}
\newcommand{\TheAuthors}{Z.K.~Birhanu, T.~Mengesha, A.J.~Salgado}

\headers{\ShortTitle}{\TheAuthors}

\title{\TheTitle}

\author{
  Zerihun Kinfe Birhanu\thanks{School of Mathematics and Statistics, Hawassa University, Ethiopia.
    (\email{zerihunk@hu.edu.et}, \url{https://www.hu.edu.et}).}
      \and
    Tadele Mengesha\thanks{Department of Mathematics, University of Tennessee, Knoxville, TN 37996, USA.
    (\email{mengesha@utk.edu}, \url{https://www.math.utk.edu/\string~mengesha/}).}
      \and
  Abner J.~Salgado\thanks{Department of Mathematics, University of Tennessee, Knoxville, TN 37996, USA.
    (\email{asalgad1@utk.edu}, \url{https://www.math.utk.edu/\string~abnersg})}
}

\ifpdf
\hypersetup{
  pdftitle={\TheTitle},
  pdfauthor={\TheAuthors}
}
\fi

\date{Draft version of \today.}

\begin{document}

\maketitle

\begin{abstract}
We consider the flow of a viscous incompressible fluid through a porous medium. We allow the permeability of the medium to depend exponentially on the pressure and provide an analysis for this model. We study a splitting formulation where a convection diffusion problem is used to define the permeability, which is then used in a linear Darcy equation. We also study a discretization of this problem, and provide an error analysis for it.
\end{abstract}


% REQUIRED

\section{Introduction}
\label{sec:intro}

In \cite{MR2292356} a hierarchy of models for fluid flow through a porous medium was developed within the context of mixture theory. It was shown that the classical Darcy's model of porous medium flow is the simplest of this hierarchy, and some extensions and variations of it are proposed and justified. Supported by experimental evidence, the reasonableness of this hierarchy has been rigorously demonstrated in \cite{MR2292356}.

One of the simplest models proposed in \cite{MR2292356} is Darcy's model but with a porosity coefficient that depends on the pressure, that is
\begin{equation}
\label{eq:Darcyp}
    \alpha(p) \bu + \nabla p = \bef, \qquad \nabla \cdot \bu = 0.
\end{equation}
Here $p$ is the pressure within the fluid, $\bu$ is its velocity, and $\bef$ represents an external force acting on the fluid. Reference \cite{MR2769053} studied this model in the case when the porosity $\alpha$ is a bounded and smooth function,  see also \cite{MR2727929}. The authors of \cite{MR2769053} also developed a heuristic analysis of the case of exponential dependence on pressure, that is
\begin{equation}
\label{eq:permeability}
  \alpha(p) = \alpha_0 \exp[ \gamma p],
\end{equation}
where $\alpha_0$ and $\gamma$ are positive parameters. The proposed formulation reduced this problem to the solution of two linear equations: a convection diffusion problem and a linear Darcy model, this formulation will be reviewed in section~\ref{sec:contprob} below. The analysis of each of the resulting discrete linear problems has been investigated in  \cite{MR2769053} but the rigorous well posedness of this strategy remained an open problem. The first goal of this work is to attempt to fill this gap. Under some assumptions on the problem data we will show that the splitting strategy is meaningful and, thus, we will use it to define a solution to our problem. Reference \cite{MR2769053} also proposed a discretization of this split formulation, and provided an error analysis for it, provided the discrete solution to the convection diffusion equation remained positive, see \cite[(4.15)]{MR2769053}. This assumption, however, was not verified and this will be the second goal of this work. Our second goal is to close the gaps in the error analysis for the proposed numerical method.

Our presentation is organized as follows. In Section~\ref{sec:prelim} we establish notation and recall some useful facts. Our problem of interest is presented in Section~\ref{sec:contprob}. The splitting formulation is introduced in Section~\ref{sub:Kirchoff}, and its analysis is presented in Section~\ref{sub:analysissplit}. We introduce a suitable, realistic, and reasonable set of assumptions on the geometry and problem data which guarantee that the split formulation of our problem is well posed. the discretization is described and analyzed in Section~\ref{sec:discrete}. Finally, some numerical illustrations of the positivity that is at the heart of our analysis  are presented in Section~\ref{sec:Numerics}.

\section{Notations and technical tools}

\label{sec:prelim}
Throughout the paper, $\Omega \subset \R^d$ with $d \in \{2,3,4\}$, is a bounded domain with Lipschitz boundary. The analysis can be extended to higher dimensions under suitable integrability and regularity of the data.
%We assume that the boundary is divided into two pieces $\Gamma_w$ and $\Gamma$. 
Whenever $X(D)$ is a normed space of functions over $D$, we indicate by $\| \cdot \|_X$ its norm. $X(D)'$ denotes the dual of $X(D)$. If the spatial domain needs to be indicated, then we will denote it by $\| \cdot \|_{X(D)}$. For $r \in [1,\infty)$, we denote the Banach space of Lebesgue $r$--integrable functions by $L^{r}(\Omega)$ with the norm
$
\|v\|_{L^r}^r = \int_{\Omega}|v|^{r} \diff x . 
$
 For $r=\infty$, $L^{\infty}(\Omega)$ represents  the space of essentially bounded measurable functions on $\Omega$, with the usual norm. For $k$ positive integer and $1\leq r\leq\infty,$ $W^{k, r}(\Omega)$ denotes the space of functions in $L^{r}(\Omega) $ whose weak partial derivatives of order up to $k$ are all in $L^{r}(\Omega)$. With the norm 
 $
 \|v\|_{W^{k,r}}^r = \|v\|^{r}_{L^{r}} + \sum_{|i|\leq k}\|\partial^{i}v\|_{L^r}^{r}, 
 $ the space 
 $W^{k, r}(\Omega)$ is a Banach space.  
For $0<s<1$, we also use the notation $W^{s,r}(\Omega)$ to denote the set of functions $v$ in $L^{r}(\Omega)$ with 
\[
|v|_{W^{s,r}}^{r}=\int_{\Omega}\int_{\Omega}{|v(y)-v(x)|^{r}\over |x-y|^{d+rs}} \diff x \diff y < \infty. 
\]
$W^{s,r}(\Omega)$ is a Banach space with its natural norm $\|v\|_{W^{s,r}}^r = \|v\|^{r}_{L^{r}}+ |v|_{W^{s,r}}^{r}$. 
For $r=2$, we set $H^{s}(\Omega)=W^{s, 2}(\Omega)$ for $s$ integer or $s\in (0, 1)$. For vector--valued functions we use boldface and the spaces of these functions are denoted, for instance,
by $\bL^{r}(\Omega).$ We also need the space $\bH(\DIV, \Omega)$ which is defined as
\[
  \bH(\DIV, \Omega)=\{\bv \in \bL^{2}(\Omega): \DIV \bv \in L^{2}(\Omega)\}. 
\]

We use several facts about Sobolev spaces. The first is the trace property, namely, owing to the fact that the boundary of $\Omega$ is Lipschitz, if $v\in H^{s}(\Omega)$ for $s\in (\tfrac12, 1]$, then the trace of $v$ on $\partial \Omega$, which we denote simply by $v|_{\partial \Omega}$, belongs to the space $H^{s-1/2}(\partial \Omega)$. Moreover, there is  constant $C$  with the estimate 
\[
\|v|_{\partial \Omega}\| _{H^{s-1/2}}\leq C \|v\|_{H^{s}}
\]
 See for instance Grisvard \cite[Theorem 1.5.1.2]{Grisvard}. 
 For $\Gamma\subseteq \partial \Omega$ with $\mathcal{H}^{d-1}(\Gamma) >0,$ we say $v\in H^{1/2}(\Gamma)$ belongs to $ H_{00}^{1/2} (\Gamma)$ if its zero--extension to $\partial \Omega$ belongs to $H^{1/2}(\partial \Omega)$. 
 
 We will also recall that vector fields in $\bH(\DIV, \Omega)$ have a well defined trace of their normal component along the boundary of $\Omega$. Namely, if $\bn(\bx)$ is the outward normal vector at $\bx\in \partial \Omega$, then for any $\bv\in \bH(\DIV, \Omega)$ we have that $\bv\cdot \bn|_{\partial \Omega}$ is in $H^{1/2}(\partial \Omega)'$ and its action  is defined  via the divergence formula
 \[
 \langle \bv\cdot \bn, u|_{\partial \Omega}\rangle_{\partial \Omega} := \int_{\Omega}  {\bf v} \cdot \nabla u \diff x  + \int_{\Omega} \DIV \bv \,u \diff x,\quad \forall u\in H^{1}(\Omega). 
 \]    
 In the event $\bv \cdot \bn|_{\partial\Omega}$ is Lebesgue integrable, then the duality pairing $ \langle\cdot, \cdot\rangle_{\partial \Omega}$ is a mere integration over $\partial \Omega$. 
 
 The second property of Sobolev spaces that will be used frequently is the fact that they embed into function spaces of higher integrability. Precisely, from Sobolev embedding theorem, we have that for $0< s\leq 1, $
 \[
 H^{s}(\Omega) \hookrightarrow L^{2^\ast}(\Omega),\quad\text{where $2^\ast = {{2d\over d-2s}}$ if $d\neq 2s$, and any $2^\ast>2$ if $d=2s.$}
 \]  
 along with the estimate: there is a universal constant $C = C(\Omega,s)$ such that 
 \[
 \|v\|_{L^{2^\ast}} \leq C\|v\|_{H^{s}}, \quad \forall v\in H^{s}(\Omega). 
 \] 
Notice that if $s=1$ and $d\leq 4$, then $2^\ast \geq 4$.
 
\section{The continuous problem}
\label{sec:contprob}
We begin by providing the exact formulation of our problem at hand. We assume that the boundary of $\Omega$ is divided into two pieces: $\Gamma_w$ and $\Gamma$, with $\mathcal{H}^{d-1}(\Gamma) \mathcal{H}^{d-1}(\Gamma_w) > 0$. The problem we are interested in reads
\begin{equation}
\label{eq:NLDarcyStrong}
  \begin{dcases}
    \alpha(p) \bu + \GRAD p = \bef, & \text{ in } \Omega, \\
    \DIV \bu = 0, & \text{ in } \Omega, \\
    p = 0, & \text{ on } \Gamma_w, \\
    \bu \cdot \bn = g, & \text{ on } \Gamma,
  \end{dcases}
\end{equation}
with the permeability function defined in \eqref{eq:permeability}. While $\bef$ and $g$ are given data, the unknowns are the velocity $\bu$ and pressure $p$ of the fluid.

\subsection{The splitting formulation}
\label{sub:Kirchoff}
Let us now recall the transformation that allowed \cite{MR2769053} to write this problem as two linear ones. This will be useful in identifying regularity requirements for the data $\bef$ and $g$. Take the first equation in \eqref{eq:NLDarcyStrong} and divide it by $\alpha$. Incompressibility then implies that
\begin{equation}
\label{eq:takediv}
  \DIV\left( e^{-\gamma p }\GRAD p \right) = \DIV( e^{-\gamma p } \bef).
\end{equation}
Define $q=e^{-\gamma p}$ and note that
\[
  \GRAD q = -\gamma e^{-\gamma p}\GRAD p,
\]
so that \eqref{eq:takediv} can be rewritten as
\[
  -\Delta q = \gamma q \DIV\bef + \gamma \bef \cdot \GRAD q.
\]
In addition, since $p=0$ on $\Gamma_w$ we have that $q=1$ there. If we assume that $\bef$ is sufficiently smooth to have a normal trace on $\Gamma$ then taking a normal trace of the first equation in \eqref{eq:NLDarcyStrong} we obtain
\[
  \alpha(p) g + \partial_n p = \bef \cdot \bn.
\]
This, for the variable $q$ means
\[
  \partial_n q + \gamma q \bef \cdot \bn = \alpha_0 \gamma g.
\]
In conclusion, for the variable $q$ we have obtained the following convection diffusion problem
\begin{equation}
\label{eq:qvarStrong}
  \begin{dcases}
    -\Delta q  - \gamma \bef \cdot \GRAD q - \gamma q \DIV\bef= 0, & \text{ in } \Omega, \\
    q = 1, & \text{ on } \Gamma_w,\\
    \partial_n q + \gamma q \bef \cdot \bn = \alpha_0 \gamma g, & \text{ on } \Gamma.
  \end{dcases}
\end{equation}
This motivates the following strategy to solve \eqref{eq:NLDarcyStrong}.
\begin{enumerate}[$\bullet$]
  \item Find $q$ that solves \eqref{eq:qvarStrong}.
  
  \item Define
  \begin{equation}
  \label{eq:alphatilde}
    \tilde\alpha(x) = \frac{\alpha_0}{q(x)}, \qquad x \in \bar\Omega.
  \end{equation}
  
  \item Find $(\bU,P)$ that solve
  \begin{equation}\label{eq:UPDarcy}
    \begin{dcases}
      \tilde \alpha \bU + \GRAD P = \bef, & \text{ in } \Omega, \\
      \DIV \bU = 0, & \text{ in } \Omega, \\
      P = 0, & \text{ on } \Gamma_w, \\
      \bU \cdot \bn = g, & \text{ on } \Gamma.
    \end{dcases}  
    \end{equation} 

\end{enumerate}
If $(\bu, p)$ solves \eqref{eq:NLDarcyStrong}, and $p\in L^{\infty}(\Omega),$ then $\tilde\alpha$ in \eqref{eq:alphatilde} will be a function that is bounded from below and above by positive numbers.  Equation \eqref{eq:UPDarcy} is now the classical linear Darcy's equation and its solution coincides with $(\bu, p)$.  The advantage of the splitting strategy is that it gives a meaning to a solution of a nonlinear system by transforming it into two linear systems when the set up leading to this transformation is applicable.     
%which can be analyzed via standard tools and that will be done in the next section. 

\subsection{Analysis of the problem}
\label{sub:analysissplit}

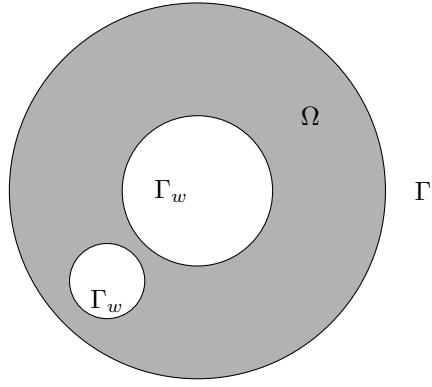
\begin{figure}
\label{fig:domain}
  \begin{center}
    \begin{tikzpicture}
      \draw[fill=white!70!black]  circle[radius = 2.5];
      \node at (3,0) {$\Gamma$};
      \draw[fill=white] circle[radius = 1] node[anchor = east] {$\Gamma_w$};
      \draw[fill=white] (-1.2,-1.2) circle[radius = 0.5] node[anchor = north] {$\Gamma_w$};
      \node at (1.5,1) {$\Omega$};
    \end{tikzpicture}    
  \end{center}
\caption{Geometry of the domain. The domain $\Omega$ is a bounded Lipschitz domain, where a finite number of strictly contained subdomains are removed. The ``exterior'' boundary is $\Gamma$, whereas the ``interior'' is $\Gamma_w$.}
\end{figure}

We now provide an analysis of the splitting strategy. In order to do so, we will operate under the following assumptions:
\begin{enumerate}[$\bullet$]
  \item \emph{Geometry.} The domain $\Omega$ is constructed as follows: Let $D \subset \R^d$ be a bounded, nonempty, Lipschitz domain. Let  $N \in \mathbb{N}$, and for $i=1, \ldots, N$,  $\calO_i \subset \R^d$ be a bounded, nonempty, Lipschitz domain. We assume that, for $i\neq j$, $\dist(\calO_i,\calO_j)>0$, and that $\cup_{i=1}^N\calO_i \Subset D$. Then $\Omega = D \setminus \overline{\cup_{i=1}^N\calO_i}$ with $\Gamma = \partial D$, and $\Gamma_w = \partial \cup_{i=1}^N\calO_i$; see Figure~\ref{fig:domain}. Essentially we are working on annuli--type domains and $\Gamma_w$ is taken to be the inside boundary. 
  
  \item \emph{Data regularity.} We assume that the problem data satisfies the following conditions.
  \begin{enumerate}[$\circ$]
    \item \emph{Permeability.} The parameters $\alpha_0$ and $\gamma$ are positive constants.
    
    \item \emph{Volume forcing.} The volume forcing term $\bef$ satisfies
    \[
      \bef \in \bL^t(\Omega), \qquad \DIV \bef \in L^{2+\delta}(\Omega), \qquad \bef\cdot\bn|_{\Gamma} \in L^{m-1}(\Gamma),
    \]
    with $t > d$; $\delta = 0$ if $d\leq 3$, and $\delta>0$ if $d=4$; and $m>d$.
    
    \item \emph{Boundary forcing.} The boundary forcing term $g$ satisfies
    \[
      g \in H^{1/2}_{00}(\Gamma)' \cap L^{m-1}(\Gamma),
    \]
    with $m>d$.
  \end{enumerate}
  
  \item \emph{Sign conditions.} We assume that the volume and boundary forcing terms satisfy the following sign conditions.
  \begin{enumerate}[$\circ$]
    \item \emph{Volume forcing.} The function $\bef$ satisfies
    \[
      \DIV \bef \leq 0, \text{ in } \Omega, \qquad \bef\cdot\bn \geq 0, \text{ on } \Gamma.
    \]

    \item \emph{Boundary forcing.} The boundary forcing term $g$ satisfies
    \[
      g \geq 0, \text{ on } \Gamma.
    \]
  \end{enumerate}
  
\end{enumerate}

We remark that the geometry assumption guarantees that the pieces of the boundary $\Gamma$ and $\Gamma_w$ are well separated. In other words, there exists $\epsilon_{0}>0$ such that for any $\epsilon \in (0, \epsilon_{0})$ we have $(\Gamma + \epsilon)\cap \Gamma_w = \emptyset$ and the domain $\Omega \setminus \overline{\Gamma + \epsilon}$ is Lipschitz. Here $(\Gamma + \epsilon) =  \cup_{x\in \Gamma} B(x, \epsilon)$.  %\AJS{Check that this is really what we mean}

We now introduce a subspace of functions in $H^{1}(\Omega)$ whose trace vanish on $\Gamma_{w}$:  
\[
  H^{1}_{w}(\Omega) = \left\{v\in H^{1}(\Omega): \  v|_{\Gamma_w} = 0 \right\}. 
\]
Our notion of a solution to \eqref{eq:NLDarcyStrong} is defined via the splitting strategy as follows. We immediately comment that, owing to the regularity conditions on $\bef$, the definition we give below makes sense.

\begin{definition}[solution]
\label{def:defofsol}
We say that the triple $(q,\bU,P) \in H^1(\Omega) \times \bL^2(\Omega) \times H^1_w(\Omega)$ is a solution to \eqref{eq:NLDarcyStrong} if $q-1 \in H^1_w(\Omega)$, 
\begin{equation}
\label{eq:convdiffweak}
  \int_\Omega \left( \GRAD q \GRAD \phi  - \gamma \bef \cdot \GRAD q \phi - \gamma q \DIV\bef \phi \right) \diff x + \gamma\int_\Gamma q \bef\cdot \bn \phi  \diff \sigma = \alpha_0 \gamma \langle g, \phi \rangle_\Gamma, 
\end{equation}
for all $\phi \in H^1_w(\Omega),$ and, with $\tilde \alpha$ defined in \eqref{eq:alphatilde},
\begin{equation}
\label{eq:linDarcy}
  \begin{dcases}
    \int_\Omega \left( \tilde \alpha \bU + \GRAD P \right) \cdot \bv \diff x = \int_\Omega \bef \cdot \bv \diff x, & \forall \bv \in\bL^2(\Omega), \\
    \int_\Omega \bU \cdot \GRAD r \diff x = \langle g, r \rangle_\Gamma, & \forall r \in H^1_w(\Omega).
  \end{dcases}
\end{equation}
\end{definition}

Let us now proceed to show that this formulation, under our imposed assumptions on the data, is well posed.  We remark that the notion of solution given in Definition \ref{def:defofsol} assumes that for given a solution $q$ of \eqref{eq:convdiffweak}, the function $\tilde \alpha$ is well defined almost everywhere in $\Omega$ and can be suitably used to solve the linear Darcy equation \eqref{eq:linDarcy}.  It is known that $\tilde \alpha$ being a bounded and strictly positive function is sufficient to demonstrate that  \eqref{eq:linDarcy} is uniquely  solvable, corresponding to appropriate data, see \cite[Theorem 2.34]{Guermond-Ern}. As a consequence, the conditions we impose on the data must ensure that not only we have a unique solution to \eqref{eq:convdiffweak} but that it also gives a $\tilde \alpha$ that is bounded and strictly positive. 

 We begin by showing existence and uniqueness of solutions for subproblem \eqref{eq:convdiffweak}. To simplify notation, we define the bilinear form
\[
  \calA_{\bef }(q,\phi) = \int_\Omega \left( \GRAD q \cdot \GRAD \phi  - \gamma \bef \cdot \GRAD q \phi - \gamma q \DIV\bef \phi \right) \diff x + \gamma\int_\Gamma q \bef\cdot \bn \phi \diff \sigma
\]
The subproblem \eqref{eq:convdiffweak} can now be rephrased, after the change of variables $z=q-1$, as: given $\bef$, and $g$ satisfying our conditions, find $z \in H^1_w(\Omega)$ such that 
\begin{equation}\label{eq:forz-weak}
  \calA_{\bef }(z,\phi) = \gamma \int_{\Omega} \DIV \bef \phi \diff x+ \gamma \langle \alpha_0 g- \bef\cdot \bn, \phi \rangle_\Gamma, \quad \forall \phi \in H^1_w(\Omega). 
\end{equation}
Well posedness of this problem is established in \cite{MR2769053} via Lax-Milgram theorem under a smallness assumption on the gradient of $\bef$. The following proposition establishes well posedness for other classes of the data $\bef$, namely those that satisfy the  regularity and sign conditions stated at the beginning of this subsection.

\begin{proposition}[boundedness and coercivity]
\label{prop:coercive}
Let $g\in H^{1/2}_{00}(\Gamma)'$, and $\bef$ satisfy our regularity and sign conditions. Then problem \eqref{eq:convdiffweak} has a unique solution.
\end{proposition}
\begin{proof}
We will show that the equivalent problem \eqref{eq:forz-weak} is well posed. Since $\bef \in \bH(\DIV,\Omega)$, and  $g$ is in  $H^{1/2}_{00}(\Gamma)'$, the map 
\[
  \phi\mapsto \gamma \int_{\Omega} \DIV \bef \phi \diff x+ \gamma \langle \alpha_0 g- \bef\cdot \bn, \phi \rangle_\Gamma,
\]
defines a bounded linear functional on $H^1_w(\Omega).$ This follows from the trace and embedding estimates discussed in Section \ref{sec:prelim}.  Thus, we only need to show that, under the given assumptions on $\bef$, the bilinear form $ \calA_{\bef }$, is bounded and coercive. Boundedness follows immediately from the regularity assumptions on $\bef$ and its divergence. Indeed, since $d\leq 4$, we have that $2^\ast \geq 4$. Consequently, Sobolev embedding shows that
\[
  \left| \gamma \int_\Omega q \DIV \bef \phi \diff x \right| \leq C \| \DIV \bef \|_{L^2} \| q \|_{L^4}\| \phi \|_{L^4} \leq C \| \DIV \bef \|_{L^2} \| \GRAD q \|_{\bL^2}\| \GRAD \phi \|_{\bL^2} .
\]
Moreover, we have that
\[
  \left| \gamma \int_\Omega \bef \cdot \GRAD q \phi \diff x \right| \leq C \| \bef \|_{\bL^t} \| \GRAD q \|_{\bL^2} \| \phi \|_{L^{2^\ast}},
\]
which together with Sobolev embedding shows the desired boundedness.

Let us now show coercivity, set $\phi = z\in H^1_w(\Omega)$ on the left hand side of \eqref{eq:forz-weak} to obtain
\begin{align*}
  \calA_\bef(z,z) &:= \int_\Omega \left (|\GRAD z|^2 - \frac\gamma2 \bef \cdot \GRAD |z|^2 - \gamma \DIV \bef |z|^2 \right)\diff x + \gamma \int_\Gamma \bef\cdot\bn |z|^2 \diff \sigma\\
    &= \int_\Omega \left( |\GRAD z |^2 - \frac\gamma2 \DIV \bef |z|^2 - \frac\gamma2 \DIV( \bef |z|^2 ) \right)\diff x + \gamma \int_\Gamma \bef\cdot\bn |z|^2 \diff \sigma\\
    &= \int_\Omega \left( |\GRAD z |^2 - \frac\gamma2 \DIV \bef |z|^2 \right)\diff x - \frac\gamma2 \int_\Gamma \bef \cdot \bn |z|^2 \diff \sigma + \gamma \int_\Gamma \bef\cdot\bn |z|^2 \diff \sigma\\
    &= \int_\Omega \left( |\GRAD z |^2 - \frac\gamma2 \DIV \bef |z|^2 \right)\diff x + \frac\gamma2 \int_\Gamma \bef \cdot \bn |z|^2 \diff \sigma\geq \int_\Omega |\GRAD z|^2 \diff x,
\end{align*}
where, in the last step, we used the sign condition on the divergence of $\bef$ and its normal trace on $\Gamma$, and that completes the proof. 
\end{proof}

We remark that after integration by parts on the left hand side of \eqref{eq:forz-weak}, we have that 
\[
  \int_\Omega \left( \GRAD z+\gamma \bef z \right)\cdot\GRAD \phi \diff x = \gamma \int_{\Omega} \DIV \bef \phi \diff x  +  \gamma \langle \alpha_0 g - \bef \cdot \bn, \phi \rangle_\Gamma, \quad \forall \phi \in H^1_w(\Omega).
\]
It then follows that $z$ is a weak solution of the mixed boundary value problem
\begin{equation}
\label{eq:forz}
  \begin{dcases}
    -\DIV (\GRAD z + \gamma \bef z) = \gamma \DIV \bef & \text{in } \Omega, \\
    z=0 & \text{on } \Gamma_w,\\
    (\GRAD z + \gamma \bef z)\cdot\bn =  \gamma(\alpha_0 g - \bef \cdot \bn)&  \text{on } \Gamma.
  \end{dcases}
\end{equation}
Having shown that this problem always has a unique solution, we will make sure that \eqref{eq:alphatilde} defines a suitable coefficient for \eqref{eq:linDarcy} to make sense. The goal is to exploit the sign condition on $g$ to guarantee that the solution $q$ is bounded from below by a positive number.  We begin by proving a regularity result for $q$. 

\begin{proposition}[regularity]
\label{prop:regularity}
Suppose that $\Omega$ satisfies our geometry assumptions, $\bef$ and $g$ satisfy our regularity assumptions, and $\bef$ satisfies our sign conditions. Then, there exists $\nu\in (0, 1)$ such that $q \in H^1(\Omega)$, the unique solution of \eqref{eq:convdiffweak} belongs to $C^{0, \nu}(\bar{\Omega})$. 
\end{proposition}
\begin{proof}
For $q \in H^1(\Omega)$, the solution to \eqref{eq:convdiffweak}, the function $z=q-1\in H^{1}_{w}(\Omega)$ solves the mixed boundary value problem \eqref{eq:forz}. It follows from local H\"older regularity of solutions of elliptic equations; see, for example, \cite[Theorem 8.24]{MR737190}, that the regularity conditions on $\bef$, guarantee the existence of $\nu\in (0, 1)$ such that $z\in C^{0, \nu}_{loc}(\Omega), $ with the estimate that for any $\Omega'\Subset\Omega$, then there exists a constant $C = C(\Omega')$ such that 
\[
\|z\|_{C^{0, \nu}(\Omega')} \leq C\left(\|z\|_{L^{2}} + \|\bef\|_{\bL^{t} }\right). 
\]
In particular, since for any $\epsilon \in (0, \epsilon_0)$, $\partial{(\Gamma + \epsilon)}\cap\Omega$ is in the interior of $\Omega, $ $z$ is H\"older continuous on $\partial{(\Gamma + \epsilon)}\cap\Omega$. 
As a consequence, since $z=0$ on $\Gamma_w$,  we may apply the global H\"older  regularity of solutions; see \cite[Theorem 8.29]{MR737190} and the discussion following this result, to conclude that for any $\epsilon \in (0, \epsilon_0)$, there exists $\nu\in (0, 1)$ (possibly different from the previous $\nu$) such that $z\in C^{ 0, \nu}(\overline{\Omega \setminus (\Gamma + \epsilon)})$. 

Next we study the regularity of $z$ near $\Gamma$ where a Robin--type boundary condition is imposed. H\"older regularity estimates for elliptic equations over Lipschitz domains with Neumann, and more generally with Robin, boundary conditions are obtained in the work of  Nittka \cite{NITTKA2011860,NittkaRobin2013Qeap}. Applying \cite[Theorem 3.7]{NittkaRobin2013Qeap} or \cite[Proposition 3.6]{NITTKA2011860}, we obtain that we can choose $\epsilon_0 >0$ such that a solution $z$ to \eqref{eq:forz} is in $C^{0, \nu}((\Gamma + \epsilon_0)\cap \Omega )$ for some $\nu\in (0, 1)$ and for any $\epsilon \in (0, \epsilon_0)$ with the estimate 
\[
\|z\|_{C^{0, \nu}((\Gamma + \epsilon) \cap \Omega)} \leq C\left(\|z\|_{L^{2}} + \|\bef\|_{\bL^{t}} +  \|g\|_{L^{m-1}} + \|\bef\cdot\bn\|_{L^{m-1}}  \right).
\]

Combining the above estimates we obtain that $z\in C^{0, \nu}(\bar\Omega)$ with the estimate that 
 \[
 \|z\|_{C^{0, \nu}(\bar\Omega)} \leq C\left(\|z\|_{L^{2}} + \|\bef\|_{\bL^{t}} +  \|g\|_{L^{m-1}} + \|\bef\cdot\bn\|_{L^{m-1}} \right),
 \]
which is what we needed to show.
\end{proof}

With the regularity result of Proposition \ref{prop:regularity} at hand, we have that, in particular, $q$ is continuous on $\bar{\Omega}$ and so it is bounded. We must, additionally, show that this function is strictly positive in $\bar\Omega$, to be able to conclude that $\tilde \alpha$, defined in \eqref{eq:alphatilde}, is an admissible coefficient. This is the content of the following result.

\begin{theorem}[positivity]
\label{thm:qpositive}
In the setting of Proposition~\ref{prop:regularity} assume, in addition, that $g = 0$ on $\Gamma$. Then, we have that the unique solution to problem \eqref{eq:convdiffweak} satisfies 
\[
  q(x) \geq q_0, \quad \forall x \in \bar\Omega
\]
for some $q_0>0$. 
As a consequence $\tilde \alpha \in C(\bar\Omega)$ and there is $\alpha_0>0$ such that $\tilde \alpha(x) \geq \alpha_0$ for all $x \in \bar\Omega$.
\end{theorem}
\begin{proof}
We prove this in two steps. First, we show that the solution $q$ is nonnegative and then applying a separate argument we show that $q$ is strictly positive. 

\noindent {\bf Step 1:} In this step we show that $q\geq 0$, using the argument developed by  Chicco in \cite{ChiccoMaurizio1970Pdmg}, adjusted to our setting, and which in turn uses Stampacchia's truncation method \cite{STAMPACCHIA}.  For $k\in \mathbb{R}, $ we begin by defining  
$
q_k(x) = \min\{q(x)-k, 0\}, 
$ and 
\begin{align*}
  \Omega(k) &= \{x\in \Omega: q_k <0\}\\
  k_0 &= \sup\{k: q_{k} \equiv 0 \ \text{in } \Omega\}.
\end{align*}
We aim to show that $q_0 \geq 0$. For that it suffices to show that $k_0 \geq 0$. We argue by contradiction and assume that  $k_0 <0$. 
We prove first that 
\begin{equation}\label{pos-omega}
\lim_{k\to k_0} |\Omega(k)| = 0. 
\end{equation}
To do so, observe that from the  definition of the set $\Omega(k)$,  as $\epsilon \downarrow 0$ we have, for every $x \in \Omega$,
$
\chi_{\Omega(k_0 -\epsilon)}(x)\to \chi_{\Omega(k_0)}(x) 
$.
As a consequence, $|\Omega(k_0)| = \lim_{\epsilon \downarrow 0}|\Omega(k_0 -\epsilon)| = 0$. In addition, one can easily show that 
as $\epsilon \downarrow 0$, 
\[
\chi_{\Omega(k_0 +\epsilon)}(x)\to \chi_{P}(x)   \quad \forall x\in \Omega,
\]
where  $P = \Omega(k_0) \cup \{ x\in \Omega:  q=k_0\}$.  In this case, we have that 
\[
\lim_{\epsilon \downarrow 0} |\Omega(k_0 +\epsilon)| =  | \{x\in \Omega: q(x)=k_0\}|. 
\]
Assume that this limit is positive. Owing to the fact that $|\Omega(k_0)|=0$ we have $k_0 - q(x) \leq 0$ for every $x \in \Omega$. Moreover, using that $q \in H^1(\Omega)$ solves \eqref{eq:convdiffweak} and $k_0$ is a constant,
\[
  \calA_\bef(k_0 - q, \phi ) = - \gamma k_0 \int_\Omega \DIV \bef \phi \diff x \leq 0, \quad \forall \phi \in H_0^1(\Omega), \ \phi \geq 0.
\]
Where, to obtain the inequality, we used the sign condition on the divergence of $\bef$. The weak minimum principle of \cite[Corollary 1]{ChiccoMaurizio1970Pdmg} then implies that, either $q(x) = k_0$ in $\Omega$ or $q(x) > k_0$ almost everywhere in $\Omega$. However, if $q(x) = k_0 < 0$ in $\Omega$, we arrive at a contradiction, as $q \in C(\bar\Omega)$ is positive (in fact $q \equiv 1$) on $\Gamma_w$. On the other hand, if $q(x)>k_0$ almost everywhere, then $| \{x\in \Omega: q(x)=k_0\}| = 0$, which is again a contradiction. In conclusion, $\lim_{\epsilon \downarrow 0} |\Omega(k_0 +\epsilon)| =0$ and \eqref{pos-omega} holds.

We now use \eqref{pos-omega} to deduce that for every $\eta>0$, there exists $ k_1 \in (k_0,0)$  such that $0<|\Omega(k_1)| < \eta$ and for all $k \leq k_1$
\[
q_k(x) = 0 \quad \forall x \in \Omega\setminus \Omega(k).   
\]
Note that for such  $k$, we have $q_k \leq 0$ in $\Omega$. Moreover, since $q=1$ on $\Gamma_w$ and $q$ is continuous, there is an open set $\calO \subset \Omega$, such that $\Gamma_w \subset \partial\calO \cap \partial \Omega$, where we have
\[
  q(x) \geq \frac12, \quad \forall x\in \bar\calO, 
\]
and therefore $ q(x)-k \geq \frac12-k >\frac12$ for all $x\in \bar\calO$. Consequently, $q_k = 0$ on $\bar\calO$ and this implies that $q_k \in H^{1}_w(\Omega)$, which makes it a suitable test function in \eqref{eq:convdiffweak}. Therefore,
\begin{align*}
  \calA_\bef(q_k, q_k) = \calA_\bef(q-k, q_k) &= \alpha_0 \gamma \int_{\Gamma} g q_k \diff \sigma - \calA_\bef(k, q_k) \\
&= \int_{\Gamma} \gamma (\alpha_0 g - k\bef\cdot \bn) q_k \diff \sigma + \gamma k \int _{\Omega}\DIV\bef q_k \diff x  \leq 0,
\end{align*}
where we also used the sign conditions on the data. Now, the coercivity of $\calA_\bef$, proved in Proposition~\ref{prop:coercive}, implies that
\[
  \| \GRAD q_k \|_{\bL^2(\Omega(k))}^2 = \| \GRAD q_k \|_{\bL^2}^2 \leq \calA_\bef(q_k,q_k) \leq 0, 
\]
so that $q_k \equiv 0$ on $\bar\Omega$ (recall that $q_k \equiv 0$ on $\Gamma_w$). However, we have arrived at a contradiction, since $q_k <0$ in $\Omega(k)$, which has positive measure. In conclusion, we must have that $k_0 \geq 0$ and therefore $q(x) \geq 0$, as we intended to show in this step.

\noindent {\bf Step 2}. Applying the weak minimum principle in \cite[Corollary 8.1]{STAMPACCHIA}, the minimum of $q$, which could possibly be zero, can only be achieved at the boundary $\Gamma$. That is, on any set $\Omega'$ compactly contained in $\Omega$ we must have, $\inf \{q(x) : x\in \Omega'\}  > 0$.

Assume now that $x_0 \in \Gamma$ is such that $0 = q(x_0) = \inf\{ q(x) : x \in \bar\Omega \}$. Since $g = 0$ in $\Gamma$, we can invoke the boundary Harnack inequality of \cite[Theorem 3.1]{Chicco2005} to assert then that, for a sufficiently small $\rho>0$,
\[
  0 \leq \max \left \{q(x): x \in \Omega \cap B(x_0,\rho) \right\} \leq C \min \left \{q(x): x \in \Omega \cap B(x_0,\rho) \right\} = 0,
\]
which is a contradiction. Therefore $q(x) > 0$ in $\bar\Omega$, as we intended to show.
\end{proof}

\begin{remark}[sharpness and positivity under other conditions]
\label{rem:otherConditionsforpositivity}
The condition $g=0$ may seem rather restrictive, as one may expect that $g\geq 0$ may be sufficient. After all, under enough smoothness of the domain $\Omega$ and solution $q$, this is all is needed to conclude the strict positivity, say, via a boundary Hopf lemma; see \cite[Lemma 3.4]{MR737190} or \cite[Section 6.4.2]{MR2597943}. However, the counterexamples of \cite{MR1941780}, which are attributed to A. Castro, show that in the case that the boundary is merely Lipschitz, as it is our case of interest here, positivity may fail at ``corner'' points.

It is possible, nevertheless, to prove strict positivity under other assumptions:
\begin{enumerate}[$\bullet$]
  \item In the case that $\Omega$ is convex, $\Gamma_w = \emptyset$ (which does not satisfy our geometry assumptions), $\bef\cdot\bn \geq f_0 > 0$ and $g\geq 0$; we can invoke the weak Harnack inequality of \cite[Lemma 3.2]{MR1941780} to conclude the strict positivity of $q$. Notice that the proof of Proposition~\ref{prop:coercive} shows that, under the strict positivity assumption on $\bef\cdot\bn$, we still have coercivity of $\calA_\bef$.
  
  \item In the case that $\Omega$, $\bef$ and $\bg$ are sufficiently smooth to guarantee that $q \in W^{2,d}(\Omega)$, and $g\geq0$; we can invoke \cite[Corollary 3.2]{LIEBERMAN2001178} to arrive at the same conclusion. \ermk
\end{enumerate}
\end{remark}

Once we know that $\tilde \alpha$ is bounded and strictly positive, the analysis of \eqref{eq:linDarcy} is standard. We summarize the well posedness of the splitting strategy in the following result. 

\begin{theorem}[existence and uniqueness]
\label{thm:ContinuousDone}
Suppose that the domain $\Omega$ satisfies our geometry assumptions; and the data $\bef$ and $g$ satisfy our regularity, and sign conditions. Under these conditions, problem \eqref{eq:NLDarcyStrong} has a unique solution in the sense of Definition~\ref{def:defofsol}.
\end{theorem}
\begin{proof}
Owing to Proposition~\ref{prop:coercive} problem \eqref{eq:convdiffweak} has a unique solution which moreover, via Proposition~\ref{prop:regularity}, is continuous and strictly positive (Theorem~\ref{thm:qpositive}). This implies that $\tilde \alpha$, defined as in \eqref{eq:alphatilde}, is a bounded and strictly positive function. The conditions of \cite[Theorem 2.34]{Guermond-Ern} are now satisfied and this implies that the linear Darcy equation \eqref{eq:linDarcy} has a unique solution.
\end{proof}

\section{Discretization}
\label{sec:discrete}

Having studied the continuous problem, we can proceed with its approximation. In addition to the conditions that guaranteed well posedness of the continuous problem we shall, to avoid unnecessary technicalities, assume that $\Omega$ is a polytope. This guarantees that $\Omega$ can be triangulated exactly. Given a  conforming, and quasiuniform triangulation of $\Omega$ (see \cite{CiarletBook} for a definition of these notions) of size $h>0$, we construct two finite element spaces $W_h \subset H^1(\Omega)$, and $\bX_h \subset \bL^2(\Omega)$. We also define $M_h = W_h \cap H^1_w(\Omega)$.

We assume that the family of pairs $(\bX_h , M_h)$ satisfies a discrete inf--sup condition: There is a constant $\beta>0$ such that for all $h>0$
\begin{equation}
\label{eq:LBB}
  \beta \| \GRAD r_h \|_{\bL^2} \leq \sup_{\bv_h \in \bX_h} \frac{ \int_\Omega \GRAD r_h \cdot \bv_h }{\| \bv_h \|_{\bL^2} }, \quad \forall r_h \in M_h.
\end{equation}
Examples of suitable spaces can be readily found in the literature \cite{CiarletBook,Guermond-Ern,MR851383,MR3097958}. To ensure positivity of discrete approximations to the variable $q$, we require that the Galerkin projection with respect to $\calA_\bef$ onto our space $M_h$ has almost optimal approximation properties in the max norm. In other words, if $w \in H^1(\Omega)$ and $w_h \in W_h$ are such that
\[
  \calA_\bef(w-w_h, \phi_h) = 0, \quad \forall \phi_h \in M_h,
\]
then
\begin{equation}
\label{eq:ThisIsWhereWeCheat}
  \| w - w_h \|_{L^\infty} \leq C |\log h| \inf \left\{ \| w - \phi_h \|_{L^\infty} : \phi_h \in W_h \right\}.
\end{equation}

\begin{remark}[max norm estimates]
The derivation of max norm error estimates for finite element approximations is a, rather technical, and underdeveloped subject. To our knowledge, most of the references that deal with this subject are only concerned with the Dirichlet problem for the Laplacian, and assume at least convexity of the domain; see for instance \cite{MR3470741} for the Laplacian, and \cite{2004.09341} for the Dirichlet problem for more general operators, but under an acuteness assumption on the triangulation. The only reference we are aware of that deals with mixed boundary conditions is \cite{MR3651062}, where convexity is also assumed, and the differential operator is the Laplacian. While we admit that this is a weakness of our analysis, we shall proceed assuming that \eqref{eq:ThisIsWhereWeCheat} holds. Another possible approach to obtain such an estimate is by deriving a discrete maximum principle, as it is detailed, for instance, in \cite[Chapter III, Sections 20, 21]{MR1115237}. This, however, imposes restrictions on the mesh.
\ermk
\end{remark}

We approximate the solution to \eqref{eq:NLDarcyStrong} with the finite element spaces that we have just described. We will say that the triple $(q_h,\bu_h,p_h) \in W_h \times \bX_h \times M_h$ is a finite element approximation of the solution to \eqref{eq:NLDarcyStrong}, in the sense of Definition~\ref{def:defofsol} if:
\begin{enumerate}[$\bullet$]
  \item The function $q_h \in W_h$ is such that $q_h -1 \in M_h$ and
  \begin{equation}
    \label{eq:convdiffdiscrete}
    \calA_\bef(q_h,\phi_h) = \alpha_0 \gamma \langle g, \phi \rangle_\Gamma, \quad \forall \phi_h \in M_h.
  \end{equation}
  
  \item We define
  \begin{equation}
  \label{eq:defofalphah}
    \tilde \alpha_h(x) = \frac{\alpha_0}{q_h(x)}, \quad x \in \bar\Omega.
  \end{equation}
  
  \item The pair $(\bu_h,p_h) \in \bX_h \times M_h$ satisfies
  \begin{equation}
  \label{eq:linDarcydiscrete}
  \begin{aligned}
    \int_\Omega \left( \tilde \alpha_h \bu_h + \GRAD p_h \right) \cdot \bv_h &= \int_\Omega \bef \cdot \bv_h, & \forall \bv \in\bX_h, \\
    \int_\Omega \bu_h \cdot \GRAD r_h &= \langle g, r_h \rangle_\Gamma, & \forall r \in M_h.
  \end{aligned}
\end{equation}
\end{enumerate}

Our main goal now is to show that the discrete problem is well posed and to study its approximation properties.

\subsection{Analysis of the discrete problem}
\label{sub:discrwellposed}

Here we show that, under similar assumptions as for the continuous problem, the discrete problem \eqref{eq:convdiffdiscrete}---\eqref{eq:linDarcydiscrete} is uniformly well posed.

\begin{theorem}[well posedness]
\label{thm:discrWellPosed}
In the setting of Theorem~\ref{thm:ContinuousDone} or Remark~\ref{rem:otherConditionsforpositivity} assume, in addition, that $h$ is sufficiently small. Then problem \eqref{eq:convdiffdiscrete}---\eqref{eq:linDarcydiscrete} has a unique solution.
\end{theorem}
\begin{proof}
By conformity, coercivity of $\calA_\bef$ is inherited to $M_h$ so that problem \eqref{eq:convdiffdiscrete} has a unique solution $q_h \in W_h$. Now, since $q_h \in W_h$ is the Galerkin approximation of $q \in H^1(\Omega)$, owing to \eqref{eq:ThisIsWhereWeCheat} we have that
\[
  \| q - q_h \|_{L^\infty} \leq c |\log h| \| q - I_h q \|_{L^\infty}
\]
where $I_h$ denotes the Lagrange interpolant. Theorem~\ref{thm:ContinuousDone} now implies that $q \in C^{0,\nu}(\bar\Omega)$ for some $\nu>0$ so that,
\[
  \| q - q_h \|_{L^\infty} \leq c h^\nu |\log h| |q|_{C^{0,\nu}},
\]
which, if $h$ is sufficiently small implies, for every $x \in \bar\Omega$,
\[
  \frac{q_0}2 \leq q(x) - ch^\nu |\log h| |q|_{C^{0,\nu}} \leq q_h(x) \leq q(x)+ ch^\nu |\log h| |q|_{C^{0,\nu}} \leq 2 \| q \|_{L^\infty}.
\]
The previous reasoning shows that the coefficient $\tilde \alpha_h$, defined in \eqref{eq:defofalphah}, is a bounded and positive function uniformly in $h$. Therefore, since the discrete inf--sup condition \eqref{eq:LBB} holds, we again invoke \cite[Theorem 2.34]{Guermond-Ern} to conclude that problem \eqref{eq:linDarcydiscrete} is uniformly well posed.
\end{proof}

\subsection{Error analysis}
\label{sub:errAnalysis}

We now proceed with the error analysis of scheme \eqref{eq:convdiffdiscrete}---\eqref{eq:linDarcydiscrete}.

\begin{theorem}[error estimates]
In the setting of Theorem~\ref{thm:discrWellPosed} we have that, if $h$ is sufficiently small,
\begin{align*}
  \| \GRAD (q-q_h) \|_{\bL^2} & \leq c \inf_{\phi_h \in W_h} \| \GRAD(q-\phi_h) \|_{\bL^2}, \\
  \| \bU - \bu_h \|_{\bL^2} + \| \GRAD (P-p_h) \|_{\bL^2} & \leq c \left( \inf_{\bv_h \in \bX_h} \| \bU - \bv_h \|_{\bL^2} + \inf_{r_h \in M_h}\| \GRAD (P-r_h) \|_{\bL^2} \right. \\ &+ \left. \| q -q_h \|_{L^\infty} \right),
\end{align*}
where, in all estimates, the constants may depend on $(q,\bU,P)$ but are independent of $h$.
\end{theorem}
\begin{proof}
The estimate on $q - q_h$ is immediate. Let us focus on the estimates between $(\bU,P)$ and $(\bu_h,p_h)$. Setting $\bv = \bv_h \in \bX_h$ and $r = r_h \in M_h$ in \eqref{eq:linDarcy} yields
\begin{align*}
  \int_\Omega \left( \tilde \alpha_h (\bU-\bu_h) + \GRAD(P-p_h) \right)\cdot \bv_h &= \int_\Omega (\tilde \alpha_h - \tilde \alpha)\bU\cdot\bv_h, & \forall \bv_h \in \bX_h, \\
  \int_\Omega (\bU - \bu_h)\cdot\GRAD r_h &= 0, & \forall r_h \in M_h.
\end{align*}

Owing to the fact that $\tilde \alpha_h$ is uniformly bounded and positive, we can invoke discrete stability to conclude, from the previous identities, that
\begin{align*}
  \| \bU - \bu_h \|_{\bL^2} + \| \GRAD (P-p_h) \|_{\bL^2} &\le c \left( \inf_{\bv_h \in \bX_h} \| \bU - \bv_h \|_{\bL^2} + \inf_{r_h \in M_h}\| \GRAD (P-r_h) \|_{\bL^2} \right. \\ &+ \left. \| \tilde\alpha -\tilde \alpha_h \|_{L^\infty} \| \bU \|_{\bL^2} \right).
\end{align*}
Notice now that, for $x \in \bar\Omega$,
\[
  |\tilde\alpha(x) -\tilde \alpha_h(x)| = \alpha_0 \left| \frac1{q(x)} - \frac1{q_h(x)} \right| \leq \frac{2\alpha_0}{q_0^2} |q(x) - q_h(x)|,
\]
which allows us to conclude.
\end{proof}

The error estimate of the previous theorem can be combined with the regularity of Proposition~\ref{prop:regularity} to obtain
\begin{align*}
  \| \bU - \bu_h \|_{\bL^2} + \| \GRAD (P-p_h) \|_{\bL^2} &\leq c \left( \inf_{\bv_h \in \bX_h} \| \bU - \bv_h \|_{\bL^2} + \inf_{r_h \in M_h}\| \GRAD (P-r_h) \|_{\bL^2} \right. \\ &+ \left. h^\nu|\log h| \right).
\end{align*}
We end by commenting that, if further regularity on $q$ can be asserted, the pointwise estimate can be improved and a higher rate can be obtained.

\section{Numerical illustrations}
\label{sec:Numerics}

Numerical illustrations of the error estimates we proved in Section~\ref{sub:errAnalysis} were presented in \cite[Section 5]{MR2769053}. Here then we confine ourselves to illustrating the positivity of the variable $q$, which is at the heart of the splitting formulation.

The computations were carried out with the help of the \texttt{FreeFem++} package \cite{MR3043640}. We used a piecewise linear discretization of the variable $q$.

\subsection{Smooth domain}
\label{sub:SmoothDomain}
In this case we consider the domain to be the annulus $\Omega = B(0,4) \setminus \overline{B(0,1)}$. We set up the forcing to be as
\[
  \bef = \kappa \left(\frac{(r-5)^2}r \frac{x}r, \frac{(r-5)^2}r \frac{y}r \right)^\intercal,  \quad \alpha_0=1, \quad \gamma = 2,
\]
where $r^2 = x^2 + y^2$, and $\kappa >0$ is to be chosen. Notice that
\[
  \DIV \bef = \kappa \left( 1 - \frac{25}{r^2} + \frac{(r-5)^2}r \right) < 0, \text{ in } \Omega, \quad \bef \cdot \bn \geq 0, \text{ on } \Gamma,
\]
so that this forcing fits within our theory. We consider two cases:
\begin{enumerate}[1.]
  \item $g = 0$ and $\kappa = \tfrac6{10}$, which fits the framework of Theorem~\ref{thm:qpositive},
  \item $g=\tfrac1{10}$ with $\kappa = 1$  that does not. See, however, Remark~\ref{rem:otherConditionsforpositivity}.
\end{enumerate}

\begin{table}
  \begin{center}
    \begin{tabular}{ccc}
      $g=0, \ \kappa = \tfrac6{10}$ & ~~ & $g=\tfrac1{10}, \ \kappa = 1$ \\
      {
        \begin{tabular}{r|r}
          \textbf{NDOFs} & $\min\{q_h(x) : x \in \overline{\Omega} \}$ \\
          \hline
          621 & -0.187745  \\
          2358 & 1.38403e-07 \\
          5247 & 2.89064e-07\\
          9549 & 3.57685e-07 \\
          14722 & 3.93091e-07 \\
          20476 & 4.12966e-07 \\
          61489 & 4.4366e-07 \\
          229441 & 4.57017e-07
        \end{tabular}
      }& &
      {
        \begin{tabular}{r|r}
          \textbf{NDOFs} & $\min\{q_h(x) : x \in \overline{\Omega} \}$ \\
          \hline
          621 & -0.478815 \\
          2358 & -0.0668046 \\
          5247 & 0.0280257 \\
          9549 & 0.028424 \\
          14722 & 0.0285943 \\
          20476 & 0.0286825 \\
          61489 & 0.0287896 \\
          229441 & 0.0288382
        \end{tabular}
      }
    \end{tabular}
  \end{center}
  \caption{Minimum value of $q_h$ for the numerical experiments of Section~\ref{sub:SmoothDomain}.}
  \label{tab:SmoothDomain}
\end{table}

Table~\ref{tab:SmoothDomain} shows the minimal value of $q_h$ as a function of the number of degrees of freedom. Notice that, as Theorem~\ref{thm:discrWellPosed} shows, positivity can only be guaranteed for a sufficiently small mesh (sufficiently large number of degrees of freedom). After that, the minimum of $q_h$ seems to stabilize at a positive value.

\subsection{Polygonal domain}\label{sub:polygon}

\begin{figure}
  \begin{center}
    \begin{tikzpicture}
      \draw[fill=white!70!black] (-3.25, -3.25)  rectangle  (3.25, 3.25);
      
      \draw[fill=white] (-0.5,-0.5) rectangle (0.5,0.5);
      \fill[white] (3.15, 3.15) rectangle (3.26, 3.26);
      
      \draw (3.25, 3.15) -- (3.15,3.15) -- (3.15,3.25);
    \end{tikzpicture}
  \end{center}
  \caption{The polygonal domain for the numerical experiment of Section~\ref{sub:polygon}}
\end{figure}
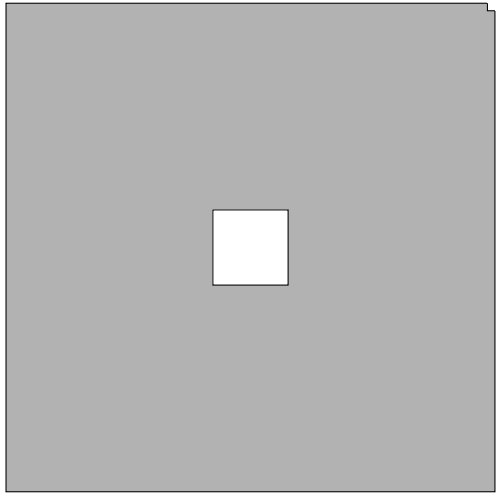

We now consider a polygonal domain. We set $a=\tfrac1{10}$, $b=\tfrac{a}5$, and $c=0.65$ and define
\[
  \Omega = (-c,c)^2\setminus\left( (-a,a)^2 \cup (c-b,c)^2 \right).
\]
The problem data is $\gamma = 2$, $\alpha_0 = 4e^\gamma$ and we set the problem data so that the exact solution is
\[
  \bu = \left( \frac{x}{x^2 + y^2} + y, \frac{y}{x^2 + y^2} - x\right)^\intercal, \qquad p = (x^2 - a^2)(y^2 - a^2).
\]

\begin{table}
  \begin{center}
    \begin{tabular}{r|r}
      \textbf{NDOFs} & $\min\{q_h(x) : x \in \overline{\Omega} \}$ \\
      \hline
      278 & 0.728008 \\
      2406 & 0.712993 \\
      9452 & 0.711648 \\
      35806 & 0.711294 \\
      25435 & 0.711323 \\
      99579 & 0.711305 \\
      404843 & 0.711387
    \end{tabular}
  \end{center}
  \caption{Minimum value of $q_h$ for the numerical experiment of Section~\ref{sub:polygon}.}
  \label{tab:polygon}
\end{table}

The minimal value of $q_h$ as a function of the number of degrees of freedom is illustrated in Table~\ref{tab:polygon}. In this case, positivity is obtained for all values of $h$.

\section*{Acknowledgments}
The core of this work was completed during the Fall of 2019, in a pre COVID-19 world, where traveling and research visits still existed. ZKB would like to thank the Department of Mathematics at the University of Tennessee Knoxville for its hospitality during his visit. ZKB also gratefully thanks NORHED HU-PhD-Math-Stat-Science project for financial support. The work of AJS has been partially supported by NSF grant DMS-1720213. TM is partially supported by NSF grant DMS-1910180.

\bibliographystyle{siamplain}
\bibliography{biblio}
\end{document}